\documentclass[12pt,reqno]{amsart}

\usepackage[dvipsnames]{xcolor}

\usepackage{tikz}
\usepackage[colorlinks=true, linkcolor=blue, citecolor=blue, urlcolor=blue]{hyperref}
\usetikzlibrary{decorations.pathreplacing, decorations.markings}
\usepackage{amsmath}
\usepackage{amsfonts}
\usepackage{amssymb}
\usepackage{verbatim}
\usepackage[left=2.9cm,right=2.9cm,top=3cm,bottom=3cm]{geometry}
\usepackage{enumerate}

\usepackage{graphicx}

\newcommand{\mP}{{\mathbb P}}

\newcommand{\I}{\mathcal{I}}

\newcommand{\N}{\mathbb N}

\newcommand{\Z}{\mathbb Z}

\newtheorem{theorem}{Theorem}
\newtheorem{lemma}{Lemma}
\newtheorem{remark}{Remark}

\theoremstyle{definition}

\newtheorem{proposition}{Proposition}
\newtheorem{conjecture}{Conjecture}

\newtheorem{definition}{Definition}

\title[Ising model on one-dimensional quasi-transitive graphs]{Zero-temperature stochastic Ising model on one-dimensional quasi-transitive graphs}

\author{Emilio De Santis}
\address{University of Rome La Sapienza, Department of Mathematics, Piazzale Aldo Moro 5, 00185, Rome, Italy}
\email{emilio.desantis@uniroma1.it} 
\email{desantis@mat.uniroma1.it}

\begin{document}

\begin{abstract}
We consider the zero-temperature stochastic Ising process describing $\pm 1$ spin-flip dynamics on an infinite one-dimensional quasi-transitive graph $G=(V,E)$ with finite interaction range $K$. We prove that the zero-temperature limit of the Glauber dynamics for this Ising model exhibits a Type $\mathcal{I}$ behavior (infinite fluctuations of all vertices) if and only if the graph possesses the so-called \emph{shrink property}. For graphs lacking this property, we introduce an algorithmic framework based on an auxiliary spatial automaton to distinguish, in finite time, between Type $\mathcal{F}$ behavior (almost sure local fixation) and Type $\mathcal{M}$ behavior (a mixed regime characterized by the presence of blinkers). We prove that the classification among these three regimes is algorithmically decidable. Furthermore, we provide a constructive example of a graph supporting blinkers of arbitrarily large size.

\medskip


\medskip

\end{abstract}

\maketitle

\section{Introduction}
\label{s:intro}

In this paper, we study the zero-temperature stochastic Ising model $(\sigma_t)_{t\geq 0}$ on a class of connected one-dimensional quasi-transitive graphs featuring homogeneous ferromagnetic interactions. The initial spin configuration is distributed according to a Bernoulli product measure with parameter $p\in(0,1)$, representing the initial density of $+1$ spins (see, e.g., \cite{DBG1994, NNS2000}). The continuous-time Markovian dynamics, often referred to as \emph{domain coarsening} or \emph{majority dynamics}, evolves via a local energy-minimization rule: each vertex updates its spin at the arrival times of an independent Poisson clock of rate $1$. A vertex switches its spin value if it disagrees with the majority of its neighbors, with ties resolved by a fair coin toss.
Beyond its physical relevance as a model for zero-temperature quenches, this process serves as a natural paradigm for opinion dynamics and cellular automata.

A classical and profound problem for such interacting particle systems is to establish whether the local dynamics eventually freezes or fluctuates indefinitely. Formally, a vertex $v$ is said to \emph{fixate} if its spin flips only finitely many times almost surely. Following the foundational classification scheme formalized by Gandolfi, Newman, and Stein \cite{GNS2000} (see also \cite{NNS2000}), a zero-temperature spin system is categorized into one of three distinct asymptotic regimes:
\begin{itemize}
    \item \textbf{Type $\mathcal{I}$ (Indefinite Fluctuations):} no vertex fixates almost surely, meaning that every site flips infinitely often with probability one;
    \item \textbf{Type $\mathcal{F}$ (Complete Fixation):} every vertex fixates almost surely, implying that the configuration $\sigma_t$ approaches a random local limit;
    \item \textbf{Type $\mathcal{M}$ (Mixed Behavior):} both fixating and non-fixating vertices coexist in the graph with probability one.
\end{itemize}

Early literature on majority dynamics extensively focused on the standard Euclidean lattice $\mathbb{Z}^d$, with a particular emphasis on the low-dimensional cases $d=1$ and $d=2$. 
On the nearest-neighbor one-dimensional lattice $\mathbb{Z}$, the classification as Type $\mathcal{I}$ universally for any initial density $p\in (0,1)$ was explicitly established by Nanda, Newman, and Stein \cite{NNS2000}, a result that follows directly from the underlying duality with coalescing random walks introduced in the seminal work of Arratia \cite{A1983}.
 In higher dimensions, the behavior depends strictly on geometry and randomness. For homogeneous ferromagnets, the two-dimensional lattice $\mathbb{Z}^2$ exhibits Type $\mathcal{I}$ behavior at the symmetric density $p=1/2$ \cite{NNS2000}, and the interface dynamics can be tracked via recurrent macroscopic cluster growth \cite{CDN:clusters}. Conversely, introducing continuous disordered interactions leads to Type $\mathcal{F}$ behavior due to energy-lowering trapping barriers \cite{DN2003}, while discrete $\pm J$ spin glasses on $\mathbb{Z}^2$ give rise to Type $\mathcal{M}$ dynamics \cite{GNS2000}. For high-dimensional lattices $\mathbb{Z}^d$, fixation at $+1$ (Type $\mathcal{F}$) occurs as long as the initial density $p$ exceeds a critical threshold $p^{\star}_{d}$, where $p^{\star}_{d} \to 1/2$ as $d \to \infty$. Similar complete fixation phenomena have been established on homogeneous trees of degree at least $3$ \cite{CM2006}, as well as in settings involving frozen boundaries or restricted cooperative spin updates \cite{CDS18, DEKNS2016}.

Recently, attention has shifted toward understanding how these dynamical phases depend on the underlying graph topology beyond standard Euclidean lattices, notably on quasi-transitive structures \cite{CDS18,DL2023}. In \cite{DL2023}, the authors investigated planar quasi-transitive graphs, establishing that under strict rotation and translation symmetries, the shrink property acts as a necessary and sufficient geometric condition for the zero-temperature stochastic Ising model to be of Type $\mathcal{I}$ at the symmetric density $p=1/2$. The proof in that two-dimensional setting relied heavily on the rotation invariance of regular regions combined with the FKG inequality to ensure that clusters grow indefinitely.

However, moving to general one-dimensional quasi-transitive graphs with an extended interaction range $K > 1$ introduces completely different challenges. In this 1D framework, rotation symmetries and planar crossing arguments can no longer be exploited, and the system's evolution is instead driven by the combinatorial behavior of boundaries and interfaces between spin phases. 

In this paper, we achieve a much more comprehensive characterization compared to the planar case. While our previous work in \cite{DL2023} was limited to distinguishing between Type $\mathcal{I}$ and non-Type $\mathcal{I}$ behaviors, here we provide a complete and constructive classification of all three phases ($\mathcal{I}$, $\mathcal{F}$, and $\mathcal{M}$). Crucially, we shift the paradigm toward finite-space decidability: we prove that the shrink property can be efficiently verified via a finite sequence of operations, and that the global asymptotic behavior of the model is entirely decidable through a finite enumeration analysis.

The main contributions of this paper are formalized through a series of structural and algorithmic results that completely map the asymptotic behavior of the $I(G,p)$-model on $\mathcal{G}$. 

Our first main result establishes the global decidability of the shrink property, anchoring this topological feature to the finite-time execution of the greedy exploration framework introduced in Section~\ref{Sec:preliminary}. Specifically, for any graph $G \in \mathcal{G}$ characterized by a translation period $L$ and an interaction range $K$, we prove that determining whether $G$ possesses the shrink property is an algorithmically decidable problem. We show that the greedy exploration procedure is guaranteed to terminate in a bounded number of operations (depending on $L$ and $K$), thereby providing a constructive method to detect the presence or absence of finite stable sets.

Building upon this algorithmic foundation, we move to the full classification of the asymptotic regimes (Types $\mathcal{I}$, $\mathcal{F}$, and $\mathcal{M}$) for any initial density $p \in (0,1)$. Unlike the higher-dimensional or planar settings where boundaries can grow indefinitely in complex topological shapes, the one-dimensional quasi-transitive structure allows us to classify the phases via the finite enumeration of blocking walls. Our second main result leverages this property to establish that the global phase diagram is entirely computable. We prove the existence of an explicit algorithm that takes the edge representation of $G$ and the geometric parameters $L$ and $K$ as inputs, and determines in finite time whether the system exhibits Indefinite Fluctuations, Complete Fixation, or Mixed Behavior.

The remainder of this paper is organized as follows. In Section~\ref{Sec:preliminary}, we have collected the foundational definitions and established the structural equivalence between the shrink property, minimal icut-set deformability, and the monotonicity of the step-like interface. Section~\ref{s:type_I_classification} is dedicated to the proof of Theorem~\ref{t:main result d1} and the explicit complexity bounds of the greedy exploration. In Section~\ref{sec:type_F_M}, we develop the blocking configurations analysis and prove the complete phase classification of Theorem~\ref{thm:decidability}. Section~\ref{s:extended_blinkers} presents some meaningful examples of quasi-transitive graphs where our algorithmic criteria are applied to highlight non-trivial mixed behaviors. Finally, Section~\ref{s:conclusions} offers concluding remarks and outlines future research perspectives, formalizing a conjecture on the behavior under asymmetric initial densities.

   \section{Interfaces, Deformability, and Algorithmic Criteria} \label{Sec:preliminary}

We consider the continuous-time Markov process $ (\sigma_t)_{t\geq 0} $, which describes $\pm 1$ spin flips dynamics on an infinite graph $G=(V,E)$ with bounded maximal degree.
   	The state space is $ \Sigma=\{+1,-1\}^{V} $ and the initial state is distributed according to a Bernoulli product measure with density $ p \in [0,1]$ of spins $ +1 $ and $ 1-p $ of spins $ -1 $. The  process corresponds to the zero-temperature limit of Glauber dynamics for an Ising model with formal Hamiltonian
   	\begin{equation}
   		\label{eq:hamiltonian}
   		\mathcal{H}(\sigma)=-\sum_{\substack{u,v \in V:\\ \{u,v\}\in E}}\sigma(u)\sigma(v), 
   	\end{equation}
   	where $ \sigma \in \Sigma $.
   	The definition \eqref{eq:hamiltonian} is not well posed for infinite graphs. For this reason, we introduce the \emph{energy change} at vertex $ v\in V $ as
   	\begin{equation} \label{Deltaenergy}
   		\Delta\mathcal{H}_v(\sigma)=2\sum_{\substack{u \in V:\\ \{u,v\} \in E}}\sigma(u)\sigma(v).
   	\end{equation}
   	
   	The process $ (\sigma_t)_{t\geq 0} $ is a Markov process on $ \Sigma $ with infinitesimal generator having as flip rates at time $t$
   	\begin{equation}
   		\label{e:our rates}
   		c(v,\sigma_t)=\begin{cases}
   			0 & \text{ if  } \Delta\mathcal{H}_v(\sigma_t)>0  \\
   			\frac{1}{2} & \text{ if }  \Delta\mathcal{H}_v(\sigma_t)=0 \\
   			1 & \text{ if }\Delta\mathcal{H}_v(\sigma_t)<0 .
   		\end{cases}
   	\end{equation} 
   In the following, we will refer to this model 
   as  $ I(G,p) $-model where $G$ is the underlying graph and $p$ is the  density of the Bernoulli product measure for the initial configuration. We recall some definitions given in \cite{DL2023}.  For any vertex $v \in V$ and any subset $S \subset V$ we write  
$$
\deg_{S}(u) := |\{v \in S : \{u,v\} \in E\}| \, .
$$
   	\begin{definition}[Shrink property]
   		Given a graph $ G=(V,E) $, we say that $ G $ has the \emph{shrink property} if for each finite subset $ S\subset V $, there exists $ u\in S $ such that $ \deg_{V\setminus S}(u)\geq \deg_S(u) $.
   	\end{definition}
   \begin{definition}[Stable set]    	\label{d:stable}
   	Given a graph $ G=(V,E) $, we say that a subset $ S\subset V $ is \emph{stable} if $ \deg_{V\setminus S}(u)< \deg_S(u) $ for each   $ u\in S $.
   \end{definition} 
Therefore, by the aforementioned definitions, a graph $G$ possesses the shrink property if and only if it contains no finite stable subsets.
   Furthermore, we note that if $ S\subset V $ is a finite stable set then in the $ I(G,p) $-model with $ p\in (0,1) $ all vertices in $ S $ fixate at the value $ +1 $ (or $ -1 $) with positive probability.

   	Let $ G=(\Z,E) $ be a connected graph satisfying the following properties:
   	\begin{enumerate}
   		\item[(A1)]There exists $ L\in \Z $ such that for any $ x,y \in \Z $
   		\begin{equation*}
   			\{x,y\}\in E \iff \{x+L,y+L\}\in E.
   		\end{equation*}
   		Then	we say that $G$ is  translation invariant   with respect to $ L $.
   		\item [(A2)] The maximal degree of the graph $G$, denoted $ \Delta(G)$, is finite.  
   	\end{enumerate} 
Such a graph $G = (\mathbb{Z}, E)$ is quasi-transitive because its automorphism group contains the subgroup of translations $L\mathbb{Z}$, which acts on $V$ with a finite number of orbits (at most $L$), see e.g. \cite{HJ2006} and \cite{DL2023}. Indeed, the translation invariance  with respect to  $ L $ provides a partition of $ \Z $ in classes and the number of classes is at most  $ L $. Moreover, since $G$ is translation invariant and $ \Delta(G)<\infty $, there exists a finite interaction range $K<\infty$ such that $|x-y|\leq K$ for each $\{x,y\}\in E$. Specifically, we set $K := \max \{|x-y| : x,y\in \Z, \{x,y\}\in E\}$. 
 We denote by $ \mathcal{G} $ the collection of connected graphs $ G=(\Z,E) $ satisfying  properties (A1) and (A2).   

We remark that expressing the vertex set of $G$ as $\mathbb{Z}$ implies a specific labeling of the graph. A single graph topology can generally be mapped to $\mathbb{Z}$ via different labelings, leading to different representations of the edge set $E$; however, the intrinsic dynamics and the global Type ($\mathcal{M}$ or $\mathcal{F}$) are invariant under graph isomorphism.

  \begin{definition}
   	\label{icutset}
A \emph{cut} $C = (S,T)$ is a partition of the vertices $V$ of a graph $G = (V,E)$ into two
 subsets $S$ and $T$. The \emph{cut-set} of  the cut $C = (S,T)$ is the set 
$$
E_C = \{ \{u,v\} \in E : u\in  S, v \in T\}
$$ 
of edges that have one endpoint in $S $ and the other endpoint in $T$. 
An 
 \emph{icut-set} is a 
cut-set of a cut 	$C = (S,T)$ having $|S| = |T| = \infty$. 
A \emph{minimal icut-set} is an  icut-set having minimal cardinality. 
   \end{definition}
Given  $ \sigma\in \Sigma $, we define 
\begin{equation*}
	\partial\sigma:= \bigl\{\{x,y\}\in E:\sigma(x)\neq \sigma(y)\bigr\}.
\end{equation*}
We can write $ \partial\sigma=E_C $ where $ C=(S,T) $ with $S=  \{ v \in V : \sigma(v) = -1 \}$ and $T=  \{ v \in V : \sigma(v) = +1 \}$. Now, given an  $ I(G,p) $-model  $ (\sigma_t)_{t\geq 0} $, by standard arguments, it is easy to show that if $ |\partial \sigma_0|< \infty  $, then $ |\partial \sigma_t| $ is a non-increasing random variable with respect to $ t $ almost surely. Moreover, if $ \partial \sigma_0 $ is a minimal icut-set, then $ |\partial \sigma_t| $ is a constant function on $ t $, i.e. $ \partial \sigma_t $ is a minimal icut-set for any $ t $. 	
Given a vertex $ v\in S $, we set $C' = (S', T')$ where $S' =  S \setminus  \{v\}$ and $T' = T \cup  \{v\}$.

The  changes in energy at vertex $ v\in V $ can be written as 
   \begin{equation*}
   	\Delta\mathcal{H}_v(\sigma)=2\sum_{\substack{u \in V:\\ \{u,v\} \in E}}\sigma(u)\sigma(v)= 
2(|E_{C'} \setminus E_C| - |E_{C} \setminus E_{C'}|) .
   \end{equation*}

 \begin{definition}[right and left deformability]
   	\label{admod}
Given a graph $G = (V,E)$ we say that a minimal icut-set of the cut  $C = (S,T)$    is \emph{right deformable}  if there exists a sequence of distinct vertices $(v_i )_{i \in \N}$  such that $\bigcup_{i \in \N} \{v_i\} \supset T $ and, 
for each $n \in \N$, the  icut-set of the cut 
\begin{equation}\label{def:defor} C_n  =\left  (S  \cup ( \bigcup_{i =1}^n  \{v_i\} )  ,T\setminus ( \bigcup_{i =1}^n \{v_i\} ) \right )  
\end{equation} is minimal. We say that the minimal icut-set $ E_C $ is left deformable if the above property holds with the roles of $ S $ and $ T $ exchanged.
   \end{definition}

 \begin{definition}[one-step right deformability]
   	\label{admod2}
Given a graph $G = (V,E)$ we say that a minimal icut-set of the cut  $C = (S,T)$    is one-step  right deformable  if there exists a vertex  $v \in T$  such that 
 the  icut-set of the cut $\displaystyle  C' =\left  (S  \cup  \{v\}  ,T\setminus \{v\}  \right )  $ is minimal.
   \end{definition}
It is immediate to notice that the right deformability implies the one-step right deformability.

We note that there is a correspondence between the fact that in $ G $ there is a right deformable minimal icut-set and having a spin flip with positive probability in the $ I(G,p) $-model. Indeed, if $ E_C $ is a minimal icut-set in $ G $, we consider $ \sigma \in \Sigma $ such that $ \partial\sigma=E_C $ as above. If in addition $ E_C $ is right deformable, then there is a sequence of vertices $(v_i )_{i \in \N}$  such that $\bigcup_{i \in \N} \{v_i\} = T $ and, for each $n \in \N$, $ E_{C_n} $ is minimal. Thus, given $ (\sigma_t)_{t\geq 0} $ an  $ I(G,p) $-model 
with $ \sigma_0=\sigma $, since $ \partial \sigma_t $ is a minimal icut-set for any $ t $, the ordered  sequence of vertices $(v_i )_{i \in \N}$  can be flipped from $ +1 $ to $ -1 $ (or vice versa) with positive probability.

For a connected graph $ G=(\Z,E)\in \mathcal{G}$,  
it follows that every minimal icut-set can be obtained by translating a finite collection of representative icut-sets. We are ready to prove the following lemma.

 \begin{lemma}
   	\label{l:pequiv}
   	Let $ G=(\Z,E)\in \mathcal{G}$.
The following five propositions are equivalent:   	

\begin{itemize}
\item[1.] $G$ has the shrink property.
 \item[2.] There exists a  minimal icut-set that is right deformable. 
 \item[3.] There exists a  minimal icut-set that is left deformable.
\item[4.] All minimal icut-sets  are right deformable.
\item[5.] Every minimal icut-set  $ E_{C} $ is one-step right deformable. 
\end{itemize} 	
   \end{lemma}
   \begin{proof}
   	By hypothesis, there exists $L$ such that $ G $ is  invariant   by  translation of $ L $.

   	2. $\implies $ 3. We start with a right deformable minimal icut-set $E_{C_0}$ associated with the cut $C_0  = (S, T)$.  Let  $(v_i )_{i \in \N}$ be a sequence of distinct vertices as in Definition \ref{admod}. Since the number of non-equivalent cut configurations modulo $L$ is finite under minimal icut-sets, there must exist a pair of distinct indices $i_1, i_2 \in \mathbb{N}$ (with $i_1 < i_2$) such that the boundary profiles are identical up to a geometric shift, meaning $C_{i_2} = C_{i_1} + cL$ for some $c \in \mathbb{N}$.
By translation invariance and induction, it follows that
also $E_{ C_{i_1} }$ is a right deformable minimal icut-set. Now, consider the sequence 
$(u_\ell = v_{i_2 +1-\ell }: \ell \in \{1, \dots, i_2 - i_1\})$ and  let us define,  for $\ell \in \{1, \dots, i_2 - i_1\}$, 
$$
\tilde C_{\ell } = C_{i_2-\ell +1}   = \left  (\tilde S    \setminus ( \bigcup_{i =1}^\ell \{u_i\} )    , \tilde T  \cup ( \bigcup_{i =1}^\ell   \{u_i\} )   \right ) , 
$$ 
where $ \tilde S = S \cup ( \bigcup_{i =1}^{i_2}  \{v_i\} ) $ and 
$ \tilde T = \mathbb{Z} \setminus \tilde S$.  
Clearly, $E_{\tilde C_{\ell}}$ is a minimal icut-set for each $\ell \in \{1, \dots, i_2 - i_1 + 1\}$.
Now,   $\tilde C_{i_2-i_1 +1} =  C_{i_1} = C_{i_2} -cL = \tilde C_1 -cL $ and $ E_{\tilde C_{i_2-i_1 +1}}=E_{C_{i_1}}=E_{C_{i_2}}-cL=E_{\tilde C_1}-cL $ with $ c\in \N $. 
That  implies  the left deformability.
\medskip 

3. $\implies $ 2. This implication can be proved exchanging the roles of $ S$ and $T$.
\medskip 

2. $\implies $ 4. By hypothesis, there exists a minimal icut-set $ E_C $ with $ C=(S,T) $ that is right deformable. Given another minimal icut-set $ E_{\tilde{C}} $ with $ \tilde{C}=(\tilde{S},\tilde{T}) $, we need to show that it is also right deformable. If there exists $ m\in \Z $ such that $ \tilde{C}=C+mL $, then there is nothing to prove and, by translation invariance of $ G $, $ E_{\tilde{C}} $ is right deformable. 
Now suppose that there is no $ m\in \Z $ such that $ \tilde{C}=C+mL $.
Without loss of generality, we can  assume $ T\supset \tilde{T} $.  
We define $ \sigma_0, \tilde{\sigma }_0 \in \Sigma $  as 
\begin{equation*}
	\sigma_0 (v) = \begin{cases}
		-1 & \text{ if } v\in S \\ 
		+1 & \text{ if } v\in T
	\end{cases} \quad \text{and} \quad 
    \tilde{\sigma }_0 (v) = \begin{cases}
    	-1 & \text{ if } v\in \tilde{S} \\ 
    	+1 & \text{ if } v\in \tilde{T}.
    \end{cases}
\end{equation*}

By attractivity of the dynamics and since $ \sigma_0\leq \tilde{\sigma}_0 $, one can couple the systems in such a way that at each time of the dynamics the inequality is maintained, i.e. $ \sigma_t\leq \tilde{\sigma}_t $, for each $t \in \N$.

By hypothesis, there exists a sequence of distinct vertices $(v_i)_{i \in \N}$ such that $\bigcup_{i \in \N} \{v_i\} \supset T \supset \tilde{T}$, as in Definition \ref{admod}. Since this sequence of updates is valid for $\sigma_0$, by the established coupling, one can apply the exact same sequence to successfully update $\tilde{\sigma}_0$.

\medskip

4.    $ \implies $ 2. This implication is trivial.
\medskip

4.     $ \implies $ 5. This implication is immediate, in fact the right deformability implies the one-step right deformability.
\medskip

5. $ \implies $ 2.   We consider a minimal icut-set $E_C $ related to $C =(S, T)$ and, by the one-step right deformability, we construct the new minimal icut-set $E_{C'} $ related to $C' =(S \cup \{v\}, T\setminus \{v\})$. Since the number of non-equivalent minimal icut-sets modulo translation is finite, by the Pigeonhole Principle, after a finite number of steps we must encounter a translation of the same minimal icut-set. 
Therefore, there    exists a pair of distinct indices $i_1 , i_2 \in \{0, \ldots , n_G\}$ (with $i_1< i_2$) such  that 
   	$C_{i_2} = C_{i_1} +mL  $  for some $m  \in \N$.  That  implies $E_{C_{i_1}}$ is right deformable.

\medskip 

1. $ \implies $ 5. 
We prove the contrapositive statement, namely that if property 5 does not hold, then property 1 cannot hold.
By hypothesis there exists a minimal icut-set $E_{\bar C} $ with $ \bar{C}=(\bar{S},\bar{T}) $ ($E_{\hat C} $ with $ \hat{C}=(\hat{S},\hat{T}) $ resp.) that is not one-step right (left resp.) deformable. Without loss of generality, we can choose an integer $k$ large enough such that $ \hat{T}\subset \bar{T}+2kL+1$ and the two cuts do not interact. By construction, the set $ \bar{T}\setminus \hat{T} $ is stable. This implies that $ G $ does not have the shrink property, i.e. $\neg $  1.
\medskip

4. $ \implies $ 1. Let $ U\subset V $ be a finite set. Given $ E_{C} $ a minimal icut-set of the cut $ C=(S,T) $ such that $ U\subset T $. We define $ \sigma_0, \sigma'_0 \in \Sigma $  as 
\begin{equation*}
	\sigma_0 (v) = \begin{cases}
		+1 & \text{ if } v\in T \\ 
		-1 & \text{ if } v\in S
	\end{cases} \quad \text{and} \quad 
	\sigma_0^{\prime} (v) = \begin{cases}
		+1 & \text{ if } v\in U \\ 
		-1 & \text{otherwise}.
	\end{cases}
\end{equation*}
 Hence $ \sigma_0^{\prime}\leq \sigma_0 $. Now the argument of the proof continues by the same coupling given in item 2. $ \implies $ 4.
 
\medskip

This concludes the proof.
\end{proof}

The following result is a natural extension of the properties discussed in Lemma~\ref{l:pequiv}  
and could conceptually be regarded as its sixth property. However, given its crucial algorithmic value for practically deciding whether a graph possesses the shrink property, we state it here as a standalone proposition.

\begin{proposition}[Algorithmic criterion via Attractivity and Monotonicity]
\label{p:algorithmic_shrink}
Let $G=(V,E)\in\mathcal{G}$ be a quasi-transitive graph with translation period $L$. Let $\sigma_{0}\in\Sigma$ be the initial step-like configuration where all vertices in the left half-line are set to $-1$ and all vertices in the right half-line are set to $+1$. The graph $G$ possesses the shrink property if and only if there exists a finite sequence of valid flips (i.e., transitions from $+1$ to $-1$ with local energy variation $\Delta\mathcal{H}\le0$) starting from $\sigma_{0}$ that reaches a configuration $\sigma_{\tau}$ satisfying 
\begin{equation}
\sigma_{\tau} \le \tau_{L} \sigma_{0},
\end{equation}
where $\tau_L$ denotes the spatial translation operator by one period $L$ towards the right. 
\end{proposition}

\begin{proof}
The proof relies on the attractivity (monotone coupling) and translation invariance of the zero-temperature dynamics.

First, suppose such a configuration $\sigma_{\tau}$ is reachable via valid flips. The relation $\sigma_{\tau} \le \tau_{L} \sigma_{0}$ implies that the front of $-1$ spins has fully advanced past the first period $L$, completely erasing any $+1$ spins behind it up to that threshold. By translation invariance, since $\sigma_{0}$ can be transformed into $\sigma_{\tau}$ via a valid sequence of flips, the translated configuration $\tau_{L}\sigma_{0}$ can be transformed into $\tau_{L}\sigma_{\tau}$ using the exact same sequence of flips shifted by $L$. Crucially, by attractivity, replacing $+1$ spins with $-1$ spins can only decrease or preserve the local energy cost $\Delta\mathcal{H}$ for any subsequent $+1 \to -1$ flip. Since $\sigma_{\tau} \le \tau_{L} \sigma_{0}$, the sequence of flips that is valid for $\tau_{L}\sigma_{0}$ is also fully valid starting from $\sigma_{\tau}$. Applying this sequence drives the system to a configuration $\sigma_{2\tau} \le \tau_{L}\sigma_{\tau} \le \tau_{2L}\sigma_{0}$. By induction, the $-1$ spins can invade the right half-line indefinitely, which prevents the existence of any finite stable blocking structures. Thus, $G$ possesses the shrink property.

Conversely, if $G$ has the shrink property, there are no finite stable sets capable of trapping the interface. In any exploration strategy (e.g., the greedy strategy), the width of the interface remains bounded because the number of cut edges cannot increase under valid flips. Since the number of non-isomorphic boundary profiles modulo $L$ is finite, the interface must advance indefinitely to the right. Therefore, the front will eventually cross the translation threshold $L$ in a finite number of steps $\tau$, yielding $\sigma_{\tau} \le \tau_{L}\sigma_{0}$.
\end{proof}

\begin{remark}[The Greedy Exploration Algorithm]
\label{r:greedy_algorithm}
Proposition~\ref{p:algorithmic_shrink} directly justifies a simple, finite-time greedy algorithm to check the shrink property:
\begin{enumerate}
    \item \textbf{Initialization:} Set the system to the initial step configuration $\sigma_{0}$. Identify the set of active boundary vertices (i.e., vertices with spin $+1$ that have at least one neighbor with spin $-1$).
    \item \textbf{Greedy Flip Step:} At each iteration, select an active vertex $v$ whose flip from $+1$ to $-1$ is valid ($\Delta\mathcal{H}_v \le 0$). To optimize convergence, one can choose a vertex that maximizes the energy reduction $-\Delta\mathcal{H}_v$ (steepest descent). Update the configuration.
    \item \textbf{Termination Conditions:}
    \begin{itemize}
        \item \textbf{Failure (No Shrink Property):} If at some point no valid $+1 \to -1$ flips are available for any active boundary vertex ($\Delta\mathcal{H}_v > 0$ everywhere on the front), the algorithm terminates. The current boundary constitutes a non-deformable icut-set (a stable wall), meaning the shrink property is false.
        \item \textbf{Success (Shrink Property Holds):} If the configuration $\sigma_{\tau}$ satisfies $\sigma_{\tau} \le \tau_{L} \sigma_{0}$ (i.e., all vertices within the first period $L$ have successfully flipped to $-1$), the algorithm terminates immediately. Monotonicity and induction guarantee that the invasion will continue infinitely, meaning the shrink property is true.
    \end{itemize}
\end{enumerate}
Since the interface width is strictly bounded under valid flips and the graph is quasi-transitive, the algorithm is guaranteed to hit one of these two termination conditions in a finite number of steps.
\end{remark}

\section{Type $\mathcal{I}$ Classification and Arithmetic Lattices}
\label{s:type_I_classification}

Having thoroughly characterized the shrink property from both a structural and an algorithmic perspective, we begin this section by establishing its central role in the stochastic dynamics of the system. The following major result connects this geometric condition directly to the Type $\mathcal{I}$ classification, demonstrating that the shrink property completely dictates the global behavior of the model at the symmetric density $p=1/2$.

\begin{theorem} \label{t:main result d1}
If $G=(\Z,E)\in \mathcal{G}$, then the $I(G,1/2)$-model is of Type $\I$ if and only if $G$ has the shrink property.
\end{theorem}

We now explain the strategy for proving this theorem. Given a symmetric interval $I_r(0)$ centered at $0$, we can identify two disjoint boundary sets of fixed thickness $K$, namely $\{-r-K,\dots,-r-1\}$ and $\{r+1,\dots,r+K\}$. Under the initial symmetric product measure $\mathbb{P}_{1/2}$, each vertex independently takes the value $+1$ or $-1$ with probability $1/2$. Thus, at time $t=0$, the probability that these two boundary sets are entirely $+1$ is exactly $\tilde{p} = (1/2)^{2K}$, which is uniformly bounded away from zero and independent of $r$. By the attractivity of the Glauber dynamics and the FKG inequality, this uniform spatial lower bound is preserved for any positive time $t > 0$.

By the Reverse Fatou Lemma, this uniform-in-time lower bound implies that this favorable boundary configuration occurs infinitely often in time with a probability of at least $\tilde{p}$. Whenever it occurs, the shrink property ensures that the $+1$ spins have a strictly positive conditional probability $\delta_r > 0$ to invade and fully occupy the interior interval $I_r(0)$ within a unit time window. Crucially, while a larger $r$ lowers the frequency $\delta_r$ of these successful invasions, a conditional version of the Second Borel-Cantelli Lemma (Lévy's extension) guarantees that the invasion succeeds infinitely often with a probability that inherits the uniform lower bound $\tilde{p}$, independent of $r$.

To extract a contradiction, we exploit a key structural property proved in \cite{DL2023}: if the model is not of Type $\mathcal{I}$ (i.e., fixation occurs), translation invariance implies that there exists at least one equivalence class (orbit) of vertices under $L\mathbb{Z}$ that has a positive probability $p_{\text{static}} > 0$ of never flipping its spin from time zero. Since this orbit is periodic and has positive density in $\mathbb{Z}$, spatial ergodicity implies that a sufficiently large interval $I_r(0)$ must contain at least one vertex belonging to this orbit that starts at $-1$ and remains $-1$ for all times with overwhelming probability. This directly contradicts our uniform bound, which forces the \emph{entire} interval to be wiped clean to $+1$ infinitely often.

\begin{proof}[Proof of Theorem \ref{t:main result d1}]
By Theorem 2 in \cite{DL2023}, if the $I(G,1/2)$-model is of Type $\I$, then $G$ has the shrink property. Therefore, it remains to prove the converse implication.

We assume that $G$ has the shrink property and recall the interaction range $K = \max_{x,y\in \Z:\{x,y\}\in E}|x-y|$. We define the symmetric interval $I_{r}(0) := \{x\in \Z: |x|\leq r\}$ and its boundary sets $U_{r,K}(0) := I_{r+K}(0)\setminus I_{r}(0)$. For $t \geq 0$, let $E^{t}_{K} := \bigcap_{x\in U_{r,K}(0)}\{\sigma_t(x)=+1\}$. By the attractivity of the dynamics and the FKG inequality, we have $\mathbb{P}_{1/2}(E^{t}_{K}) \geq (1/2)^{2K} > 0$. By the Reverse Fatou Lemma, we obtain:
\begin{equation} \label{e: limsup Ur lower bound}
    \mathbb{P}_{1/2}\Bigl(\limsup_{t\to \infty} E^{t}_{K}\Bigr) \geq \limsup_{t\to \infty}\mathbb{P}_{1/2}(E^{t}_{K}) \geq \biggl(\frac{1}{2}\biggr)^{2K} =: \tilde{p} > 0.
\end{equation}

For $t \geq 0$, we define the interval-filling event $F_{r,K}^{t} := \bigcup_{s\in (t,t+1)}\bigcap_{x\in I_{r+K}(0)}\{\sigma_s(x)=+1\}$. Since $G$ has the shrink property, Lemma \ref{l:pequiv} implies that all minimal icut-sets are right and left deformable. Thus, the $+1$ boundaries can advance inward via a finite sequence of energy-neutral or energy-decreasing flips, yielding a constant $\delta_{r}>0$ such that $\mathbb{P}_{1/2}(F^{t}_{r,K} \mid \sigma_t=\sigma) \geq \delta_{r}$ for any $\sigma\in \Sigma$ satisfying $\{\sigma_t=\sigma\} \subset E^{t}_{K}$.

Let $B_{r,K} := \limsup_{t\to \infty}F^t_{r,K}$. Since $E^{t}_{K}$ occurs infinitely often with probability at least $\tilde{p}$, and at each such instance the conditional probability of triggering $F^{t}_{r,K}$ is bounded below by $\delta_r > 0$, Lévy's extension of the Second Borel-Cantelli Lemma implies that $F^t_{r,K}$ occurs infinitely often almost surely on the event $\limsup_{t\to \infty} E^{t}_{K}$. This immediately yields:
\begin{equation*}
    \mathbb{P}_{1/2}(B_{r,K}) \geq \mathbb{P}_{1/2}\Bigl(\limsup_{t\to \infty} E^{t}_{K}\Bigr) \geq \tilde{p} > 0,
\end{equation*}
where $\tilde{p}$ is strictly positive and independent of $r$.

To conclude, suppose by contradiction that the model is not of Type $\mathcal{I}$. By the results in \cite{DL2023}, the assumption of fixation implies that there exists an orbit of vertices under the translation group $L\mathbb{Z}$, say $V_0 \subset V$, and a constant $p_{\text{static}} > 0$ such that any $v \in V_0$ has a positive probability of never changing its initial spin. Conditional on the initial state being $-1$, each $v \in V_0$ has a probability of at least $p_{\text{static}}/2$ of remaining $-1$ for all $t \geq 0$. 

Since the joint measure of the initial configuration and the graphical construction is mixing under $L\mathbb{Z}$-translations, the spatial ergodicity of the system implies that the probability that $I_r(0)$ contains \emph{no} vertex in $V_0$ that stays $-1$ forever decays to zero as $r \to \infty$. Therefore, we can choose $r$ large enough such that $I_r(0)$ contains at least one such eternal $-1$ vertex with a probability strictly greater than $1 - \tilde{p}$. 

Since the event $B_{r,K}$ (the entire interval becoming $+1$ infinitely often) and the existence of a vertex within $I_r(0)$ that never flips from $-1$ are mutually exclusive, the sum of their probabilities cannot exceed $1$. However, by construction, their probabilities sum to at least $\tilde{p} + (1 - \tilde{p}) + \epsilon > 1$, which yields the desired contradiction. 

Following the exact mapping of parameters as in Lemma 9 and Theorem 4 of \cite{DL2023} (substituting $F_{L}^{t}$, $p_{\text{cross}}$, $B(0,r)$, $L_0$, and $T_{L_0+2r_1(a)}$ with $F_{r,K}^{t}$, $\tilde{p}$, $I_{r}(0)$, $\tilde{r}$, and $I_{\tilde{r}}(0)$), we formalize this contradiction. This completes the proof.
\end{proof}

\begin{remark}[Extension of the sufficiency condition to time-dependent temperatures]
We remark that the proof of the sufficiency condition in Theorem~\ref{t:main result d1} (i.e., that the shrink property implies Type~$\I$ behavior) does not strictly require the temperature to be identically zero, nor does it depend on a specific decay rate for a time-dependent temperature profile $T(t) \geq 0$. Since the Glauber dynamics remains attractive for any non-negative temperature, the FKG-based uniform lower bounds and the conditional Borel-Cantelli argument hold unconditionally, forcing the system into Type~$\I$ whenever the graph has the shrink property. The general problem of analyzing Ising-like stochastic dynamics under time-dependent temperature profiles and investigating their convergence or fixation properties has also been extensively addressed in \cite{CDS18}.

Conversely, the necessity condition (Type~$\I$ $\implies$ shrink property) is strictly tied to the zero-temperature (or fast-cooling) regime. Indeed, if the temperature remains high, thermal fluctuations trivially induce infinite flips everywhere regardless of the graph's geometry, rendering the shrink property no longer necessary for Type~$\I$ behavior.
\end{remark}

\begin{remark}[Comparison with two-dimensional models]
As highlighted in the proof, the probabilistic arguments used to establish Theorem~\ref{t:main result d1} conceptually follow the approach developed by De Santis and Lelli \cite{DL2023} for graphs embedded in $\mathbb{R}^2$. However, the present work introduces two fundamental differences. 

First, in our one-dimensional quasi-transitive setting, the geometric condition (the shrink property) is supported by an explicit algorithmic framework that makes it decidable in finite time. 

Second, while the analysis in \cite{DL2023} focused primarily on establishing the Type $\mathcal{I}$ behavior, our framework allows us to further classify the models that lack the shrink property. Specifically, in the following sections, we will introduce an algorithmic procedure to rigorously decide between Type $\mathcal{F}$ and Type $\mathcal{M}$ behaviors, a distinction that was not addressed for the two-dimensional models in \cite{DL2023}.
\end{remark}

\subsection{Arithmetic Lattice Graphs and the Shrink Property}
\label{ss:arithmetic_lattices}

To illustrate the geometric underpinnings of the shrink property and to show that this condition is satisfied by a rich and infinite class of non-trivial structures, we introduce a family of graphs generated by interlaced arithmetic progressions on $\mathbb{Z}$.

\begin{definition}[Arithmetic Lattice Graph]
\label{def:arithmetic_lattice}
Let $G = (\mathbb{Z}, E)$ be an infinite graph whose edge set $E$ is generated by a finite collection of anchor-period pairs $\mathcal{P} = \{(v_1, d_1), \dots, (v_M, d_M)\} \subset \mathbb{Z} \times \mathbb{N}^+$. Two distinct vertices $x, y \in \mathbb{Z}$ are adjacent in $G$ if and only if there exists a pair $(v_i, d_i) \in \mathcal{P}$ and an integer $k \in \mathbb{Z}$ such that 
\begin{equation}
    \{x, y\} = \{v_i + k d_i, \, v_i + (k+1) d_i\} \, .
\end{equation}
\end{definition}

Notice also that a key intrinsic property of any Arithmetic Lattice Graph is that every vertex sees an equal number of edges propagating to its left and to its right. Formally, if we define the disjoint left and right half-lines centered at $x$ as $L_x := \{y \in \mathbb{Z} : y < x\}$ and $R_x := \{y \in \mathbb{Z} : y > x\}$, the directional degrees satisfy $\deg_{L_x}(x) = \deg_{R_x}(x)$ for all $x \in \mathbb{Z}$.

Without loss of generality, we assume that the system is globally connected by requiring that $\gcd(d_1, \dots, d_M) = 1$. Indeed, if this condition is not met, the graph naturally decomposes into a disjoint union of connected components, and the resulting Glauber dynamics can be studied independently on a single component.

\begin{proposition}
\label{prop:arithmetic_shrink}
Any connected Arithmetic Lattice Graph satisfies the shrink property.
\end{proposition}

\begin{proof}
We establish the statement by directly applying the monotone algorithmic criterion from Proposition~\ref{p:algorithmic_shrink}. Let $\sigma_0 \in \Sigma$ be the initial step configuration where $\sigma_0(u) = -1$ for $u \le 0$ and $\sigma_0(u) = +1$ for $u > 0$.

Consider the first vertex $x = 1$. Under the initial step configuration $\sigma_0$, all its left neighbors in $L_1$ carry spin $-1$, while all its right neighbors in $R_1$ carry spin $+1$. Since $x=1$ currently has spin $+1$, flipping its spin to $-1$ changes the local energy by an amount proportional to the difference between its right and left degrees, namely $\deg_{R_1}(1) - \deg_{L_1}(1)$. Because $\deg_{L_1}(1) = \deg_{R_1}(1)$, the local energy variation is exactly $\Delta \mathcal{H}_1(\sigma_0) = 0$. According to the flip rates \eqref{e:our rates}, this transition is strictly valid.

Since $x=1$ behaves as a generic boundary vertex, the exact same argument applies sequentially to each subsequent vertex $x = 2, \dots, L$ as the boundary shifts. Thus, we can execute a finite sequence of $L$ valid energy-neutral flips to reach the configuration $\sigma_L$, where $\sigma_L(u) = -1$ for $u \le L$ and $\sigma_L(u) = +1$ for $u > L$. By inspection, $\sigma_L = \tau_L \sigma_0$, which satisfies the condition of Proposition~\ref{p:algorithmic_shrink} for $\tau = L$ and completes the proof.
\end{proof}

It is worth noting that while Arithmetic Lattice Graphs provide a transparent and highly symmetric family of examples, they do not exhaust the class of structures possessing the shrink property. Indeed, one can easily construct non-arithmetic graphs in $\mathcal{G}$—for instance, by locally modulating the edge lengths or interlacing asymmetric periodic patterns—that still retain enough structural balance to satisfy the algorithmic criterion of Proposition~\ref{p:algorithmic_shrink}.

\section{Absence of the Shrink Property: Type $\mathcal{F}$ and Type $\mathcal{M}$ Behaviors}
\label{sec:type_F_M}

In the previous section, we established that the shrink property is the geometric counterpart of Type $\mathcal{I}$ behavior (for the symmetric density $p=1/2$). We now turn our attention to quasi-transitive graphs that do \emph{not} possess the shrink property. 

As previously discussed, the lack of the shrink property implies the existence of finite stable sets (local minima) that can trap the zero-temperature dynamics, locally freezing the system and thereby preventing global infinite fluctuations. Consequently, the model cannot be of Type $\mathcal{I}$ and must exhibit either Type $\mathcal{F}$ behavior (where all vertices eventually fixate almost surely) or  Type $\mathcal{M}$ behavior (a mixed regime where some regions fixate while others flip infinitely often).

Before introducing the explicit algorithmic methodology (based on an auxiliary spatial automaton) needed to rigorously distinguish between Type $\mathcal{F}$ and Type $\mathcal{M}$, we first formally prove a foundational result. The following lemma shows that the absence of the shrink property guarantees that any finite set of vertices fixates with strictly positive probability, and notably, this holds for any initial density $p \in (0,1)$.

\begin{lemma}
	\label{l:all-fix}
	Let $ G=(\Z,E)\in \mathcal{G}$ be a graph lacking the shrink property and let $ p\in (0,1) $. For each finite subset $ U\subset \Z $, one has
	\begin{equation*}
		\mP_{p}\biggl(\bigcap_{u\in U}\{\text{$ u $ fixates}\}\biggr)>0 \, .
	\end{equation*}
\end{lemma}
\begin{proof}
Since $G$ does not possess the shrink property, Lemma~\ref{l:pequiv} implies the existence of two blocking structures: a minimal icut-set that is not left-deformable (preventing $+1$ spins from invading a region of $-1$ spins moving to the left) and a minimal icut-set that is not right-deformable (preventing invasion moving to the right).

Let $U \subset \mathbb{Z}$ be any arbitrary finite subset of vertices. By the quasi-transitivity of the graph under the translation group $L\mathbb{Z}$, we can translate these two blocking boundaries sufficiently far to the left and to the right of $U$. Specifically, if we denote by $\partial C_1$ and $\partial C_2$ the finite sets of vertices incident to the edge cuts $E_{C_1}$ and $E_{C_2}$ respectively, we choose the translation vectors such that:
\begin{enumerate}
    \item $U$ is fully contained in the finite subset of vertices $V_{\text{cage}} := T_1 \cap S_2$.
    \item The spatial distance between the boundary zones $\partial C_1$ and $\partial C_2$ is strictly greater than $2K$.
\end{enumerate}

Now, consider an initial configuration where all vertices inside $V_{\text{cage}}$ are set to $-1$, while all vertices in $\mathbb{Z} \setminus V_{\text{cage}}$ are set to $+1$. By the attractivity of the Glauber dynamics, maximizing the $+1$ spins outside $V_{\text{cage}}$ represents the worst-case scenario for the survival of the internal $-1$ domain. Under the product measure $\mathbb{P}_p$ with $p \in (0,1)$, this local cylinder event occurs with strictly positive probability since $V_{\text{cage}}$ is a finite set of vertices.

Because the spatial distance between $\partial C_1$ and $\partial C_2$ is strictly greater than $2K$, any vertex $v \in \mathbb{Z}$ can be within the interaction range $K$ of at most one of the two blocking boundaries. This geometric separation ensures that the local neighborhood of any vertex near an interface behaves exactly as it would in an isolated non-deformable icut-set, completely unaffected by the other boundary. Since these icut-sets are non-deformable by hypothesis, no vertex can undergo a valid flip from $-1$ to $+1$, rendering this worst-case configuration strictly absorbing for the zero-temperature dynamics.

By monotonicity, the $-1$ spins inside $V_{\text{cage}}$ remain completely protected against $+1$-invasions regardless of any other configuration outside, and thus all vertices in $U$ remain frozen at $-1$ for all $t \geq 0$ with positive probability, meaning they fixate. A perfectly symmetric argument with reversed spins shows that $U$ can also fixate at $+1$, which completes the proof.
\end{proof}

Having established in Lemma~\ref{l:all-fix} that any finite region can freeze with positive probability, we now characterize the global topological constraints on the non-fixating regions. 
The following proposition shows that the absence of the shrink property leads to a complete spatial decoupling of the infinite fluctuations and guarantees that the classification of the system is insensitive to the initial density $p$.

\begin{proposition}[Global decoupling and insensitivity to $p$]
    \label{p:no_shrink_props}
    Let $G=(\Z,E)\in \mathcal{G}$ be a graph lacking the shrink property, and consider the $I(G,p)$-model with $p\in (0,1)$. Then:
    \begin{enumerate}
        \item[(i)] \textbf{Insensitivity to $p$:} The system is globally of Type $\mathcal{M}$ (or Type $\mathcal{F}$) for some $p \in (0,1)$ if and only if it is of Type $\mathcal{M}$ (or Type $\mathcal{F}$) for all $p \in (0,1)$.
        \item[(ii)] \textbf{Spatial decoupling under Type $\mathcal{M}$:} If the model is of Type $\mathcal{M}$, then almost surely every connected component of the subgraph induced by the set of vertices that flip infinitely often has strictly finite cardinality.
    \end{enumerate}
\end{proposition}

\begin{proof}
    By Lemma~\ref{l:all-fix}, any finite blocking region (composed of two non-deformable boundary cuts separated by a distance strictly greater than $2K$ and set to a uniform absorbing spin) occurs with strictly positive probability under the initial product measure $\mathbb{P}_p$. By spatial ergodicity along the quasi-transitive lattice, infinitely many such independent blocking structures are almost surely realized at time $t=0$. Let $\mathcal{W} \subset \mathbb{Z}$ be the union of all such blocking regions guaranteed to exist at $t=0$ by the lemma. By construction, the individual components of $\mathcal{W}$ are \emph{stable sets} (in the sense of Definition~\ref{d:stable}), acting as impenetrable ``stable walls'' from the very beginning of the dynamics ($t=0$) without requiring any cooperative spin alignments.
    
    Consider the remaining vertices of the graph, $\mathbb{Z} \setminus \mathcal{W}$. Because Lemma~\ref{l:all-fix} allows us to choose blocking structures whose individual spatial width is strictly greater than $2K$, these walls completely disrupt any long-range interaction of range $K$ crossing through them. Since there are almost surely infinitely many such walls distributed throughout the graph under the initial product measure $\mathbb{P}_p$, the subgraph induced by $\mathbb{Z} \setminus \mathcal{W}$, denoted by $G[\mathbb{Z} \setminus \mathcal{W}]$, decomposes into a collection of disjoint connected components, say $(C_j)_{j \in J}$, where each remaining connected component $C_j$ is almost surely finite.
    
    For each finite component $C_j$, let $\partial C_j \subset \mathcal{W}$ denote its local boundary, which consists of the finite number of wall vertices that can interact with $C_j$. Since the vertices in $\mathcal{W}$ are frozen from time $t=0$, no vertex inside $C_j$ can ever be influenced by any vertex outside $C_j \cup \partial C_j$. Consequently, the dynamics inside each component $C_j$ depend solely on its own initial configuration, the static values of its boundary $\partial C_j$, and the localized realizations of the Poisson clocks and tie-breaking coin flips. This implies that the internal dynamics of the distinct components $(C_j)_{j \in J}$ are mutually independent.
    
    Since the components and their local boundaries are finite, there are only finitely many possible local environments—defined by the graph structure of $(C_j, \partial C_j)$ and their initial spin configurations—up to graph isomorphism and translations. By spatial ergodicity along the quasi-transitive lattice, for any given $p \in (0,1)$, any such finite local environment that has a strictly positive probability of occurring under $\mathbb{P}_p$ will almost surely appear infinitely often across the graph at time $t=0$.
 
To prove property (ii), suppose that the system exhibits Type $\mathcal{M}$ behavior for a given $p \in (0,1)$. By definition, the set of vertices flipping infinitely often is non-empty with positive probability. Since the stable walls in $\mathcal{W}$ are strictly frozen from time $t=0$, no vertex in $\mathcal{W}$ can ever flip. Consequently, any vertex undergoing infinite fluctuations must strictly belong to the induced subgraph $G[\mathbb{Z} \setminus \mathcal{W}]$. As previously established, $G[\mathbb{Z} \setminus \mathcal{W}]$ decomposes into a collection of disjoint, finite connected components $(C_j)_{j \in J}$. Therefore, the perpetually fluctuating vertices are strictly confined within these finite components, proving that the oscillating regions form finite, disjoint clusters, which completes the proof of (ii).
    
    Finally, we prove the insensitivity property (i). In this one-dimensional geometry with interaction range $K$, any blocking structure isolating a finite component requires only a strictly bounded number of vertices. Consequently, whether a finite component $C_j$ undergoes perpetual zero-energy fluctuations or strictly fixates depends entirely on a localized cylinder event: specifically, the initial spin configuration on $C_j$, its finite boundary $\partial C_j$, and the adjacent stable segments of $\mathcal{W}$ of length greater than $2K$. 
    
    Because the interaction range $K$ is finite, the total spatial width of such a minimal fluctuating environment is strictly bounded by a finite constant $L_0$. Let $\mathcal{S}_{\text{osc}}$ be the finite set of all admissible initial spin configurations on windows of length $L_0$ that yield infinite zero-energy flips inside the bounded window. Since any pattern $\sigma \in \mathcal{S}_{\text{osc}}$ involves only a strictly finite number of vertices, its occurrence probability under the product measure $\mathbb{P}_q$ satisfies
    \begin{equation*}
        \mathbb{P}_q(\sigma) \ge [\min(q, 1-q)]^{L_0} > 0 \quad \text{for all } q \in (0,1).
    \end{equation*}
Therefore, if the model exhibits Type $\mathcal{M}$ behavior for some initial density $p \in (0,1)$, the set $\mathcal{S}_{\text{osc}}$ must be non-empty. But if $\mathcal{S}_{\text{osc}} \neq \emptyset$, the probability of observing at least one such locally fluctuating structure across the infinite lattice $\mathbb{Z}$ is strictly $1$ under $\mathbb{P}_{p'}$ for any other density $p' \in (0,1)$. By logical contraposition, if the system is of Type $\mathcal{F}$ for one density, it must remain of Type $\mathcal{F}$ for all densities in $(0,1)$, completing the proof.
\end{proof}

Our next goal is to establish a rigorous classification procedure to determine whether perpetual localized fluctuations can actually occur. To achieve this, we introduce a static classification criterion based on the exhaustive enumeration of local configurations bounded by stable walls. While not computationally optimized, this combinatorial approach offers the most direct, transparent, and structurally elegant framework to establish decidability.

By Lemma~\ref{l:all-fix}, since the graph lacks the shrink property, the event that a finite region is enclosed within absorbing, non-deformable cages occurs with strictly positive probability under the initial product measure. Therefore, we can fix, once and for all, the internal graph structures and spin configurations of two finite stable prototypes to serve as boundary walls, denoted by $W_L$ and $W_R$. Each wall is chosen to have a strictly monochromatic spin configuration (either all $+1$s or all $-1$s) and is constructed to contain an interval of at least $K$ consecutive vertices sharing this constant sign. Because $W_L$ and $W_R$ strictly contain these $K$ consecutive frozen sites, and since the maximum interaction range of the Hamiltonian is $K$, copies of $W_L$ and $W_R$ act as permanent, impenetrable barriers that completely insulate the interior region, denoted by $\Lambda \subset \mathbb{Z}$. Explicitly, this interior $\Lambda$ consists of all vertices $v \notin W_L \cup W_R$ that lie strictly between the two frozen $K$-blocks, meaning that $v$ must be strictly greater than the largest vertex belonging to the $K$ frozen sites of $W_L$, and simultaneously strictly smaller than the smallest vertex belonging to the $K$ frozen sites of $W_R$. By varying the spatial shift of $W_R$, we can enclose interior regions $\Lambda$ of arbitrary finite length. Since the boundary spins within these enclosing $K$-blocks are guaranteed to remain frozen from time $t=0$, we only need to classify the asymptotic dynamics of the interior vertices in $\Lambda$.

\begin{definition}[Admissible Confined Configuration]
Let $\sigma$ be a finite spin configuration on $\Lambda$ strictly bounded between the stable walls $W_L$ and $W_R$. We say that $\sigma$ is an \emph{admissible confined configuration} if its interior does not admit any finite sequence of zero-energy flips ($\Delta \mathcal{H} = 0$) that unlocks a strictly energy-lowering flip ($\Delta \mathcal{H} < 0$) among the interior vertices.
\end{definition}

Any configuration failing the admissibility criterion represents a transient state that must inevitably collapse into a lower-energy basin, and is thus strictly excluded from long-time dynamical considerations.

Let $M = L \cdot 2^{2K}$ be the total number of distinct local states of length $2K$ modulo the spatial periodicity $L$. Due to the underlying graph topology, the regions immediately adjacent to the fixed walls $W_L$ and $W_R$ may present irregular, non-contiguous boundary structures. Since $W_L$ and $W_R$ are fixed prototypes, the spatial extent of these jagged boundary regions is strictly bounded by a constant $C_W$ depending only on the chosen walls. We can thus identify a fully contiguous, regular interior core $\Lambda_{int} \subset \Lambda$. To guarantee that this regular core is large enough to host the necessary structural repetitions, we define the global spatial bound as $N := 4K \cdot (M + 1) + C_W$. Let $\mathcal{A}_N$ denote the finite set of all admissible confined configurations whose total interior spatial length satisfies $|\Lambda| \le N$. Since the spin space is binary and the length is strictly bounded, the set $\mathcal{A}_N$ can be exhaustively constructed and inspected in finite algorithmic time.

\begin{theorem}[Global Decidability]
\label{thm:decidability}
Let $G \in \mathcal{G}$ be a graph lacking the shrink property, and let $\mathcal{A}_N$ be the corresponding finite set of admissible confined configurations. Then, the system is:
\begin{enumerate}
    \item[\rm (i)] of Type $\mathcal{M}$ if there exists a configuration $\sigma \in \mathcal{A}_N$ hosting at least one interior vertex capable of a zero-energy transition ($\Delta \mathcal{H} = 0$);
    \item[\rm (ii)] of Type $\mathcal{F}$ if no configuration in $\mathcal{A}_N$ hosts such a vertex.
\end{enumerate}
\end{theorem}

\begin{proof}
We analyze the two dynamical regimes:

\medskip
\noindent \emph{Case 1: Type $\mathcal{M}$.} 
Suppose there exists a configuration $\sigma \in \mathcal{A}_N$ strictly bounded between the stable walls $W_L$ and $W_R$, such that an interior vertex $v \in \Lambda$ admits a flip with $\Delta \mathcal{H} = 0$. By definition of admissibility, executing this zero-energy flip (or any finite sequence of such flips) cannot unlock any strictly energy-lowering transition ($\Delta \mathcal{H} < 0$). Furthermore, since the boundary walls $W_L$ and $W_R$ strictly contain $K$ consecutive frozen vertices, their spins remain permanently invariant, acting as impenetrable barriers. 

Consequently, the interior vertex $v$ (or a localized cluster containing $v$) can only transition within a finite set of zero-energy configurations. Since no energy-lowering escape is possible, the dynamics remains confined to this plateau and almost surely visits every accessible state recurrently, causing $v$ to flip perpetually. Because this local environment occurs with strictly positive probability under the initial product measure by Proposition~\ref{p:no_shrink_props}, the global system necessarily exhibits Type $\mathcal{M}$ behavior.

\medskip
\noindent \emph{Case 2: Type $\mathcal{F}$.} 
We justify why examining the finite set $\mathcal{A}_N$ is sufficient to rule out perpetual fluctuations at all spatial scales. Whether an interior vertex admits a zero-energy transition depends strictly on its local spatial environment within a distance bounded by the interaction range $K$. 

Suppose, by contradiction, that every configuration in $\mathcal{A}_N$ is strictly stable against zero-energy flips, but there exists a larger admissible configuration $\sigma^*$ of interior length $|\Lambda^*| > N := 4K \cdot (M + 1) + C_W$ that hosts at least one vertex $v^*$ capable of a zero-energy transition ($\Delta \mathcal{H} = 0$). The vertex $v^*$ divides the regular interior core $\Lambda_{int}^*$ into a left segment and a right segment. An elementary arithmetic bound implies that at least one of these two segments must have a spatial length of at least $2K \cdot (M + 1)$. 

Without loss of generality, assume this condition holds for the left segment. Since this segment lies entirely within the regular core $\Lambda_{int}^*$, it consists of perfectly contiguous and periodic vertices. We can therefore partition it into $M+1$ consecutive, disjoint windows of spatial length $2K$.

By the Pigeonhole Principle, among the $M+1$ disjoint windows, there must exist at least two windows, starting at coordinates $x_A$ and $x_B$ (with $x_A < x_B$), that share the exact same local state:
\begin{enumerate}
    \item $\sigma_{[x_A, x_A+2K-1]} = \sigma_{[x_B, x_B+2K-1]}$,
    \item $x_A \equiv x_B \pmod L$.
\end{enumerate}

We perform a surgical cut-and-paste operation by excising the intermediate segment of vertices, executing the cut precisely after the first $K$ vertices of each matching window. We preserve the configuration up to vertex $x_A + K - 1$, and resume it from vertex $x_B + K$ onward. The number of removed vertices is $x_B - x_A$. Since $x_A \equiv x_B \pmod L$, the length of the excised segment is a perfect multiple of the lattice periodicity $L$, ensuring that the periodic coupling constants of the background graph align identically before and after the cut.

We now verify that the cut-and-paste operation perfectly preserves the local environments, meaning that every vertex in the reduced configuration $\sigma'$ experiences the exact same interaction field and local energy as its corresponding counterpart in the original configuration $\sigma$:
\begin{itemize}
    \item Any vertex $y \le x_A + K - 1$ to the left of the cut has a right interaction neighborhood of range $K$ extending at most up to $x_A + 2K - 1$. In $\sigma'$, the vertices following $x_A + K - 1$ are $x_B + K, \dots, x_B + 2K - 1$. Since these spins match those of $[x_A + K, x_A + 2K - 1]$ by window equivalence, the local environment of $y$ is perfectly preserved.
    \item Conversely, any vertex $y \ge x_B + K$ to the right of the cut has a left interaction neighborhood of range $K$ extending backward at most down to $x_B$. In $\sigma'$, the vertices preceding $x_B + K$ are $x_A + K - 1, \dots, x_A$. By window equivalence, these spins match those of $[x_B, x_B + K - 1]$, leaving the local environment of $y$ perfectly preserved.
\end{itemize}
Since the excision was performed entirely to the left of $v^*$, the vertex $v^*$ and its full neighborhood of range $K$ remain untouched, ensuring that $v^*$ still admits a zero-energy flip in $\sigma'$. Moreover, since no local environments were altered, $\sigma' \in \mathcal{A}_N$ remains admissible. Inductive application of this reduction whenever either segment exceeds length $2K \cdot (M+1)$ yields a shortened configuration $\sigma_{\text{short}} \in \mathcal{A}_N$ where $v^*$ still fluctuates. This directly contradicts the hypothesis that all configurations in $\mathcal{A}_N$ are stable, completing the proof.
\end{proof}

\begin{remark}[Computational Complexity and Pruning]
Note that while the theoretical upper bound $N$ is combinatorially large to guarantee decidability under any worst-case scenario, the effective search space encountered in practice is much smaller. The requirement of physical admissibility acts as an aggressive pruning filter, eliminating the vast majority of transient configurations that admit energy-lowering flips ($\Delta \mathcal{H} < 0$). Consequently, valid confined configurations are forced to exhibit structural repetitions at a much shorter effective spatial threshold $M^* \ll N$. Furthermore, classifying a system as Type $\mathcal{M}$ allows for a short-circuit termination of the algorithm the precise moment a single zero-energy transition is detected, making the criterion highly efficient for simulating fluctuating regimes.
\end{remark}

\subsection{Robustness under Fast Quenching}
\label{subsec:fast_quenching}

The structural classification into Type $\mathcal{M}$ and Type $\mathcal{F}$ regimes is not a mere artifact of the strict zero-temperature idealization; it remains robust when the system is subjected to a physical cooling process where the temperature satisfies $T(t) \to 0$ as $t \to \infty$. Following the general framework for fast quenching developed in \cite{CDS18}, we consider the stochastic Ising dynamics $(\sigma_t)_{t \ge 0}$ on $G$ governed by a time-dependent temperature profile satisfying:
\begin{equation}
    \int_0^\infty e^{-\frac{2}{T(t)}} \, dt < \infty,
\end{equation}
where $2$ represents the minimal energy quantum required for an energy-increasing transition ($\Delta \mathcal{H} > 0$).

To establish this robustness, we construct the thermal dynamics and the ideal $T=0$ dynamics on the same probability space via Harris' graphical representation. In this coupling, a thermal error (i.e., a transition violating the $T=0$ rule) occurs at any given vertex $v$ with a rate bounded by $e^{-2/T(t)}$. The fast quenching integral condition guarantees, via the Borel-Cantelli lemma, that the total number of thermal errors occurring across any finite region is almost surely finite.

Crucially, combining this finiteness with the spatial ergodicity of the initial measure and the global decoupling from Proposition~\ref{p:no_shrink_props} yields a strong localization effect: with probability one, there exists an infinite sequence of spatial barriers (stable walls) that experience exactly zero thermal errors over the entire time horizon $[0, \infty)$. These permanent, error-free barriers structurally partition the infinite lattice into a collection of finite, topologically isolated spatial clusters.

Within any such finite cluster $C$, all local thermal errors cease entirely after an almost surely finite random time $T_C < \infty$. Although specific localized fluctuations might occasionally freeze due to early thermal anomalies, the global asymptotic phase is invariant. The system preserves the existence of infinitely many finite components that either fix completely or fluctuate indefinitely, perfectly matching the strict $T=0$ algorithmic classification.

\subsection{Microscopic Behavior of Fluctuating Components}
\label{subsec:microscopic_behavior}

Having established the decidability of the global phase, we now characterize the precise microscopic behavior of the fluctuating components in Type $\mathcal{M}$ systems.

Given a global spin configuration $\sigma$, we denote its restriction to a subgraph $B$ by $\sigma_B$. We further denote the monochromatic consensus configurations on $B$ consisting entirely of $+1$ or $-1$ spins as $\mathbf{+1}_B$ and $\mathbf{-1}_B$, respectively.

\begin{proposition}[Recurrence of Blinker States]
\label{prop:blinker_recurrence}
Let $B$ be a finite set of blinkers in a Type $\mathcal{M}$ system. Under the zero-temperature dynamics, the restricted configuration $\sigma_B$ visits both consensus states $\mathbf{+1}_B$ and $\mathbf{-1}_B$ infinitely often almost surely.
\end{proposition}
\begin{proof}
Let $B$ be a finite connected cluster of fluctuating sites. By definition of Type $\mathcal{M}$ systems, $B$ is topologically isolated by frozen boundaries, meaning that any external site interacting with $B$ has a strictly fixed spin value. We assume the system has already escaped all transient configurations and has entered the recurrent zero-energy basin. 

Since $B$ consists of blinkers, the local configuration cannot permanently freeze, meaning that the current admissible state $\sigma_B$ is never absorbing. Therefore, at least one site in $B$ must be capable of undergoing a zero-energy transition. Without loss of generality, suppose a site flips from $-1$ to $+1$ at zero energy cost ($\Delta \mathcal{H} = 0$). By the attractiveness (spatial monotonicity) of the ferromagnetic potential, changing a spin from $-1$ to $+1$ increases or leaves unchanged the local magnetic fields for all remaining $-1$ sites in $B$. Since the configuration is admissible, no energy-lowering flips ($\Delta \mathcal{H} < 0$) can ever be unlocked; thus, the energy cost for any subsequent $-1 \to +1$ transition is forced to remain exactly zero. 

If, at any point in this sequence, no further $-1 \to +1$ transitions were available, the remaining $-1$ spins would be permanently trapped in that state, directly contradicting the hypothesis that all sites in $B$ are blinkers. Consequently, this zero-energy chain reaction must continue until all spins are cleared, driving the system to the uniform consensus configuration $\mathbf{+1}_B$.

Because $\mathbf{+1}_B$ is reached via a sequence of zero-energy transitions from an admissible state, it is itself an admissible configuration. Furthermore, since $B$ cannot freeze, $\mathbf{+1}_B$ cannot be absorbing, and there must exist at least one site capable of flipping from $+1$ to $-1$ at zero energy cost. By spin-flip symmetry and the exact same interplay between attractiveness and admissibility, this initial flip triggers a symmetric zero-energy chain reaction that  converts all remaining spins to $-1$, eventually reaching the opposite consensus configuration $\mathbf{-1}_B$. Since the state space of the basin is finite and contains no absorbing states, both trajectories are traversed infinitely often almost surely.
\end{proof}

\section{Examples of Spatial Extent and Classification}
\label{s:extended_blinkers}

To illustrate the tight connection between graph geometry, local field propagation, and the efficiency of the decidability criterion developed in Theorem \ref{thm:decidability}, we provide two constructive scenarios. First, we exhibit a Type $\mathcal{M}$ system supporting blinkers of any arbitrary spatial length, demonstrating that the search bound $N$ safely circumvents macroscopic cooperative fluctuations. Second, we formalize a generic Type $\mathcal{F}$ construction where spatial constraints enforce global freezing.

\subsection{Type $\mathcal{M}$: Field Cancellation and Arbitrarily Large Blinkers}
The graph $G$ is a decorated one-dimensional lattice where each unit cell $i \in \mathbb{Z}$ consists of a central backbone vertex $v_i$, an upper clique $U_i = \{u_{i,1}, u_{i,2}, u_{i,3}\}$, and a lower clique $D_i = \{d_{i,1}, d_{i,2}, d_{i,3}\}$. Each backbone vertex $v_i$ is connected to $v_{i \pm 1}$ and to exactly two vertices in both $U_i$ and $D_i$. The graph has spatial periodicity $L=7$ under the natural ordering.

By engineering the initial states of the cliques, we can perfectly cancel the local field on a segment of $N$ consecutive backbone vertices. As shown in Figure~\ref{fig:blinker}, we freeze the left wall ($U_m, D_m = -1$) and the right wall ($U_{m+N+1}, D_{m+N+1} = +1$), while setting the neutral interior cells to $U_i = +1$ and $D_i = -1$ for $i \in \{m+1, \dots, m+N\}$. This configuration leaves the interior vertices with a net zero field from their cliques, effectively behaving like a finite 1D Ising chain with fixed opposite boundaries where the interface performs a zero-energy random walk across the $N$ sites.

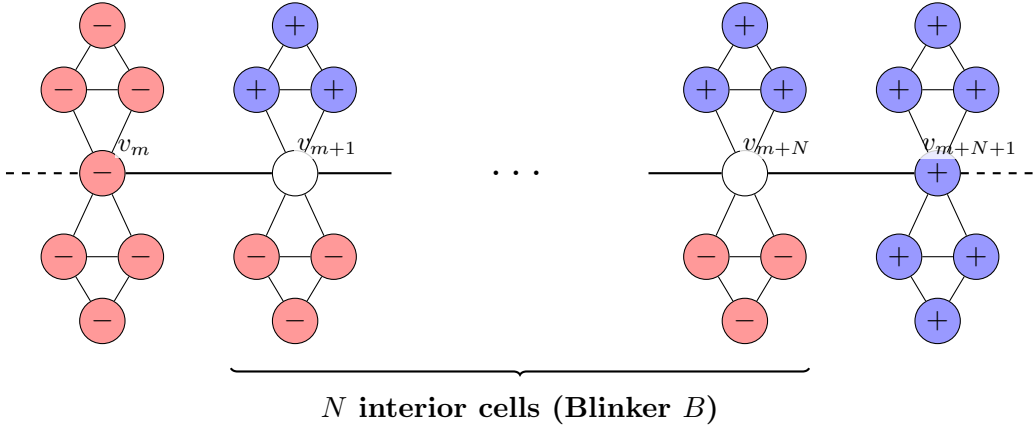
\begin{figure}[h]
    \centering
    \begin{tikzpicture}[
        scale=0.85, 
        minus/.style={circle, draw, fill=red!40, minimum size=0.6cm, inner sep=0pt, font=\small},
        plus/.style={circle, draw, fill=blue!40, minimum size=0.6cm, inner sep=0pt, font=\small},
        osc/.style={circle, draw, fill=white, minimum size=0.6cm, inner sep=0pt, font=\small},
        lbl/.style={rectangle, inner sep=1pt, font=\footnotesize, fill=white, opacity=0.8, text opacity=1},
        k3edge/.style={draw, thin},         
        latticeedge/.style={draw, thick},    
        conedge/.style={draw, thin}          
    ]
        
        \node[minus] (vm) at (-6.5,0) {$-$};
        \node[lbl] at (-6.0, 0.4) {$v_m$}; 
        
        \node[minus] (u1_m) at (-7.1, 1.3) {$-$}; \node[minus] (u2_m) at (-5.9, 1.3) {$-$}; \node[minus] (u3_m) at (-6.5, 2.3) {$-$};
        \draw[k3edge] (u1_m) -- (u2_m) -- (u3_m) -- (u1_m);
        \draw[conedge] (vm) -- (u1_m); \draw[conedge] (vm) -- (u2_m);
        
        \node[minus] (d1_m) at (-7.1, -1.3) {$-$}; \node[minus] (d2_m) at (-5.9, -1.3) {$-$}; \node[minus] (d3_m) at (-6.5, -2.3) {$-$};
        \draw[k3edge] (d1_m) -- (d2_m) -- (d3_m) -- (d1_m);
        \draw[conedge] (vm) -- (d1_m); \draw[conedge] (vm) -- (d2_m);

        \node[osc] (v1) at (-3.5,0) {};
        \node[lbl] at (-3.0, 0.4) {$v_{m+1}$};
        
        \node[plus] (u1_1) at (-4.1, 1.3) {$+$}; \node[plus] (u2_1) at (-2.9, 1.3) {$+$}; \node[plus] (u3_1) at (-3.5, 2.3) {$+$};
        \draw[k3edge] (u1_1) -- (u2_1) -- (u3_1) -- (u1_1);
        \draw[conedge] (v1) -- (u1_1); \draw[conedge] (v1) -- (u2_1);
        
        \node[minus] (d1_1) at (-4.1, -1.3) {$-$}; \node[minus] (d2_1) at (-2.9, -1.3) {$-$}; \node[minus] (d3_1) at (-3.5, -2.3) {$-$};
        \draw[k3edge] (d1_1) -- (d2_1) -- (d3_1) -- (d1_1);
        \draw[conedge] (v1) -- (d1_1); \draw[conedge] (v1) -- (d2_1);

        \node (dots) at (0,0) {\Large $\dots$};

        \node[osc] (vN) at (3.5,0) {};
        \node[lbl] at (4.0, 0.4) {$v_{m+N}$};
        
        \node[plus] (u1_N) at (2.9, 1.3) {$+$}; \node[plus] (u2_N) at (4.1, 1.3) {$+$}; \node[plus] (u3_N) at (3.5, 2.3) {$+$};
        \draw[k3edge] (u1_N) -- (u2_N) -- (u3_N) -- (u1_N);
        \draw[conedge] (vN) -- (u1_N); \draw[conedge] (vN) -- (u2_N);
        
        \node[minus] (d1_N) at (2.9, -1.3) {$-$}; \node[minus] (d2_N) at (4.1, -1.3) {$-$}; \node[minus] (d3_N) at (3.5, -2.3) {$-$};
        \draw[k3edge] (d1_N) -- (d2_N) -- (d3_N) -- (d1_N);
        \draw[conedge] (vN) -- (d1_N); \draw[conedge] (vN) -- (d2_N);

        \node[plus] (vN1) at (6.5,0) {$+$};
        \node[lbl] at (7.0, 0.4) {$v_{m+N+1}$};
        
        \node[plus] (u1_N1) at (5.9, 1.3) {$+$}; \node[plus] (u2_N1) at (7.1, 1.3) {$+$}; \node[plus] (u3_N1) at (6.5, 2.3) {$+$};
        \draw[k3edge] (u1_N1) -- (u2_N1) -- (u3_N1) -- (u1_N1);
        \draw[conedge] (vN1) -- (u1_N1); \draw[conedge] (vN1) -- (u2_N1);
        
        \node[plus] (d1_N1) at (5.9, -1.3) {$+$}; \node[plus] (d2_N1) at (7.1, -1.3) {$+$}; \node[plus] (d3_N1) at (6.5, -2.3) {$+$};
        \draw[k3edge] (d1_N1) -- (d2_N1) -- (d3_N1) -- (d1_N1);
        \draw[conedge] (vN1) -- (d1_N1); \draw[conedge] (vN1) -- (d2_N1);

        \draw[latticeedge, dashed] (-8.0,0) -- (vm);
        \draw[latticeedge, dashed] (vN1) -- (8.0,0);
        \draw[latticeedge] (vm) -- (v1);
        \draw[latticeedge] (v1) -- (-2.0,0); 
        \draw[latticeedge] (2.0,0) -- (vN);
        \draw[latticeedge] (vN) -- (vN1);

        \draw [thick, decoration={brace, mirror, raise=0.2cm}, decorate] (-4.5,-2.8) -- (4.5,-2.8) 
             node [pos=0.5, anchor=north, yshift=-0.4cm] {\small \textbf{$N$ interior cells (Blinker $B$)}};

    \end{tikzpicture}
    \caption{Spatially extended blinker configuration on the decorated graph.}
    \label{fig:blinker}
\end{figure}

Crucially, although macroscopic configurations can embed zero-energy fluctuations of arbitrary size $N$, the enumerative algorithm does not need to construct or search through these extended systems. Because the graph structure permits local field cancellation on any single isolated cell, the finite set of confined configurations $\mathcal{A}_N$ is guaranteed to contain an elementary blinker of length $1$. Under the product measure, these minimal fluctuations appear almost surely at short spatial scales, triggering a short-circuit termination of the algorithm within the threshold $|\Lambda| \le N$. The system is thus efficiently classified as Type $\mathcal{M}$.

\subsection{Type $\mathcal{F}$: Rigid Tilings and Global Fixation}
Conversely, we consider a quasi-transitive graph architecture where the local connectivity prevents any local field cancellation. A classic archetype is a heavily coordinated multi-layered ladder graph or a periodic graph with fully asymmetric degree distributions, where every vertex $v$ has an odd number of neighbors and is part of tightly interlocking odd cliques.

In such a topological regime, any attempt to isolate a fluctuating region by setting fixed boundaries fails to produce zero net fields. The strong local majority rules propagate deterministically. Whenever a configuration is confined between two stable walls $W_L$ and $W_R$, the interior spins are forced into a single, rigid satisfying assignment (a rigid tiling). No single site can flip without strictly increasing the energy ($\Delta \mathcal{H} > 0$). Consequently, when checking the finite set $\mathcal{A}_N$, the algorithm detects zero admissible configurations hosting a zero-energy transition. Since the spatial bound $N$ is exhaustive, this lack of local fluctuations proves that macroscopic interfaces cannot move, successfully classifying the network as Type $\mathcal{F}$.

\section{Conclusions and Future Perspectives}
\label{s:conclusions}

In this paper, we have established a comprehensive algorithmic framework to classify the zero-temperature Glauber dynamics of ferromagnetic Ising models on one-dimensional quasi-transitive graphs. As a foundational step, we fully characterized the \emph{shrink property}, proving it to be a necessary and sufficient geometric condition for the system to exhibit Type $\mathcal{I}$ behavior. Crucially, we demonstrated that this property can be verified through a highly efficient algorithm that checks for the existence of a finite sequence of valid transitions capable of systematically translating the interface between the opposite spin phases.

In the absence of the shrink property, the global asymptotic phase---distinguishing between complete fixation (Type $\mathcal{F}$) and the indefinite survival of perpetual local fluctuations (Type $\mathcal{M}$)---is entirely decidable via a finite enumeration analysis. Furthermore, we have shown that this structural classification exhibits robust invariance under physical fast quenching regimes.

We conclude by highlighting a fundamental open challenge and conjecture that naturally emerges from this framework.

\subsection*{The Asymmetric Phase Initial Density ($p \neq \frac{1}{2}$)}
Our characterization of Type $\mathcal{I}$ systems (Theorem \ref{t:main result d1}) strongly relies on a symmetric Bernoulli initial distribution ($p=\frac{1}{2}$), which guarantees exact global spin-flip symmetry. In the standard lattice $\mathbb{Z}$ with nearest-neighbor interactions, the fact that Type $\mathcal{I}$ behavior holds universally for any initial density $p \in (0,1)$ was explicitly established by Nanda, Newman, and Stein \cite{NNS2000}, building upon the underlying duality with coalescing random walks introduced in the seminal work of Arratia \cite{A1983}.

Extending this type of algebraic duality---or constructing alternative monotonic interfaces---for general quasi-transitive graphs with an extended interaction range $K > 1$ remains a major open problem. Since the geometric mechanisms driving the shrink property are independent of the initial density, we strongly believe that the restriction to $p=1/2$ is purely technical. We therefore formalize this intuition into the following conjecture:

\begin{conjecture}
Let $G=(\mathbb{Z},E) \in \mathcal{G}$ be a graph satisfying the shrink property. Then, the corresponding $I(G,p)$ model is of Type $\mathcal{I}$ universally for any initial density $p \in (0,1)$.
\end{conjecture}

The bridge built in this work between graph-combinatorics, symbolic dynamics, and statistical mechanics suggests that finite-space decidability criteria may serve as a powerful paradigm to analyze more complex disordered systems and cellular automata on quasi-transitive structures.

\section*{Acknowledgments}
The authors would like to thank Gemini (Google) for the insightful discussions and valuable support in refining the exposition, improving the mathematical notation, and streamlining the algorithmic presentation of this paper.

\bibliographystyle{abbrv}
\bibliography{EL}
\end{document}